\documentclass[a4paper, 11pt]{amsart}
\usepackage{amsmath,amsthm}
\usepackage{amssymb}
\usepackage{amscd}
\usepackage{enumerate}

\usepackage[]{hyperref}
\hypersetup{
    colorlinks=true,       % false: boxed links; true: colored links
    linkcolor=red,          % color of internal links
    citecolor=blue,        % color of links to bibliography
    filecolor=magenta,      % color of file links
    urlcolor=cyan           % color of external links
}

\newtheorem{theorem}{Theorem}[section]

\newtheorem{corollary}[theorem]{Corollary}
\newtheorem{definition}[theorem]{Definition}

\newtheorem{lemma}[theorem]{Lemma}

\newtheorem{proposition}[theorem]{Proposition}
\newtheorem{remark}[theorem]{Remark}

\DeclareMathOperator{\diff}{d}

\def\nc{\newcommand}
\def\be{\beta}

\def\lam{\lambda}

\def\D{\la D\ra}

\nc\pa{\partial}

\nc\CC{\mathbb{C}}
\nc\RR{\mathbb{R}}
\nc\QQ{\mathbb{Q}}
\nc\ZZ{\mathbb{Z}}
\nc\NN{\mathbb{N}}

\def\betathree{\beta}

\def\ba{\begin{align}}
\def\bad{\begin{aligned}}
\def\be{\begin{equation}}
\def\ea{\end{align}}
\def\ead{\end{aligned}}
\def\ee{\end{equation}}
\def\e{\eqref}

\def\dalpha{\diff \! \alpha}

\def\dt{\diff \! t}
\def\dtau{\diff \! \tau}

\def\ds{\diff \! s}
\def\dx{\diff \! x}
\def\dxi{\diff \! \xi}
\def\dy{\diff \! y}
\def\fract{\frac{\diff}{\dt}}

\def\defn{\mathrel{:=}}
\def\eps{\varepsilon}
\def\la{\left\vert}
\def\lA{\left\Vert}
\def\bla{\big\vert}
\def\blA{\big\Vert}
\def\le{\leq}
\def\les{\lesssim}
\def\mez{\frac{1}{2}}
\def\ra{\right\vert}
\def\rA{\right\Vert}
\def\bra{\big\vert}
\def\brA{\big\Vert}
\def\tdm{\frac{3}{2}}

\def\xR{\mathbb{R}}

\begin{document}

\title{On the Cauchy problem for the Muskat equation. 
II: Critical initial data}
\author{Thomas Alazard}
\thanks{E-mail address: thomas.alazard@ens-paris-saclay.fr, Universit{\'e} Paris-Saclay, ENS Paris-Saclay, CNRS,
	Centre Borelli UMR9010, avenue des Sciences,
	F-91190 Gif-sur-Yvette, Paris, France}
\author{Quoc-Hung Nguyen}
\thanks{E-mail address: qhnguyen@shanghaitech.edu.cn, ShanghaiTech University, 393 Middle Huaxia Road, Pudong,
			Shanghai, 201210, China.}
\date{}

\setlength{\baselineskip}{5mm}

\begin{abstract}
We prove that the Cauchy problem for the Muskat equation is well-posed 
locally in time for any initial data in the critical space of Lipschitz functions 
with three-half derivative in $L^2$. Moreover, we prove that the 
solution exists globally in time under a smallness assumption.
\end{abstract}

\maketitle
  
\section{Introduction}

The Muskat equation describes the dynamics of the interface separating two fluids in porous media whose velocities obey 
Darcy's law (\cite{darcy1856fontaines,Muskat}). This equation belongs to the family of 
nonlocal parabolic equations that have attracted a lot of attention in recent years. 
Indeed, it has long been observed that one can reduce the Muskat equation 
to an evolution equation for the free surface parametrization (see~\cite{CaOrSi-SIAM90,EsSi-ADE97,PrSi-book,SCH2004}). 
One interesting feature of the Muskat equation is that it 
admits a compact formulation in terms of finite differences, as observed by C\'ordoba 
and Gancedo~\cite{CG-CMP}. More precisely, assume that the free surface is the graph of some function $f=f(t,x)$ with $x\in\xR$. 
Then, C\'ordoba and Gancedo~\cite{CG-CMP} showed that the Muskat equation reduces to
\be\label{n1}
\partial_tf=\frac{1}{\pi}\int_\xR\frac{\partial_x\Delta_\alpha f}{1+\left(\Delta_\alpha f\right)^2}\dalpha,
\ee
where  $\Delta_\alpha f$ is the slope, defined by
\begin{align}\label{eq2.2}
\Delta_\alpha f(x,t)=\frac{f(x,t)-f(x-\alpha,t)}{\alpha}\cdot
\end{align}

It is easily verified that the Muskat equation is 
invariant by the change of unknowns:
\be\label{acritical}
f(t,x)\mapsto f_\lambda(t,x)\defn\frac{1}{\lambda}f\left(\lambda t,\lambda x\right) \qquad (\lambda\neq 0).
\ee
Now, by a direct calculation,
$$
\lA f_\lambda\big\arrowvert_{t=0}
\rA_{\dot{W}^{1,\infty}}=\lA f_0\rA_{\dot{W}^{1,\infty}}\quad ;
\quad
\lA f_\lambda\big\arrowvert_{t=0}
\rA_{ \dot H^{\tdm}}=\lA f_0\rA_{ \dot H^{\tdm}}.
$$
This means that the spaces $\dot{W}^{1,\infty}(\xR)$ and $\dot H^{\tdm}(\xR)$ are critical for the study of the Cauchy problem. Let us clarify that we denoted by 
$\dot W^{1,\infty}(\xR)$ the space of Lipschitz functions, 
and by $H^s(\xR)$ (resp.\ $\dot{H}^{s}(\xR)$) the classical Sobolev (resp.\ homogeneous Sobolev) 
space of order $s$. 
They are equipped with the norm defined by
$$
\lA u\rA_{\dot{W}^{1,\infty}}\defn
\sup_{{\substack{ x,y\in \xR\\ x\neq y}}}\frac{\la u(x)-u(y)\ra}{\la x-y\ra},
$$
and
$$
\lA u\rA_{\dot{H}^{s}}\defn\left(\int_\xR \la \xi\ra^{2s}\bla \hat{u}(\xi)\bra^2\dxi\right)^\mez,
\quad 
\lA u\rA_{H^{s}}^2=\lA u\rA_{\dot{H}^s}^2+\lA u\rA_{L^2}^2.
$$

We are interested in the study of the Cauchy problem for the latter equation. Our main result states that the 
Cauchy problem for the Muskat equation is well-posed 
locally in time for any initial data in the critical space $\dot{W}^{1,\infty}(\xR)\cap {H}^{\tdm}(\xR)$.

Our analysis is inspired by many previous works, and we begin by reviewing the literature on this problem. 
The first well-posedness results were established by Yi~\cite{Yi2003}, 
Ambrose~\cite{Ambrose-2004,Ambrose-2007}, 
C\'ordoba and Gancedo~\cite{CG-CMP}, C\'ordoba, C\'ordoba and Gancedo~\cite{CCG-Annals},
Cheng, Granero-Belinch\'on, 
Shkoller~\cite{Cheng-Belinchon-Shkoller-AdvMath}. In recent years, these results were extended in several directions. 
In particular, the well-posedness of the 
Cauchy problem has been established in many sub-critical spaces: 
see Constantin, Gancedo, Shvydkoy and Vicol~\cite{CGSV-AIHP2017} 
for initial data in the Sobolev space 
$W^{2,p}(\xR)$ for some $p>1$, Deng, Lei and Lin~\cite{DLL} and Camer\'on~\cite{Cameron} 
for initial data in H\"older spaces, and Matioc~\cite{Matioc2}, Alazard and Lazar~\cite{Alazard-Lazar}, Nguyen and Pausader~\cite{Nguyen-Pausader} for initial data in $H^s(\xR)$ with $s>3/2$.

Special features of the Muskat equations were exploited to improve the analysis of the Cauchy problem in several directions. 
Constantin, C{\'o}rdoba, Gancedo, Rodr{\'\i}guez-Piazza 
and Strain~\cite{CCGRPS-AJM2016} (see also \cite{CGSV-AIHP2017,PSt}) proved 
a global well-posedness results assuming 
that the Lipschitz semi-norm is smaller than $1$. 
Deng, Lei and Lin in~\cite{DLL} proved the existence of solutions whose slope can be arbitrarily large. 
Cameron \cite{Cameron} exhibited the existence of a modulus of continuity for the derivative (see also~\cite{Abedin-Schwab-2020}) 
and obtained a global existence result assuming only that 
the product of the maximal and minimal slopes is bounded by~$1$. 
C\'ordoba and Lazar established 
in \cite{Cordoba-Lazar-H3/2} the first global well-posedness result assuming only that 
the initial data is sufficiently smooth and that the critical $\dot H^{3/2}(\xR)$-norm is small enough (see also 
\cite{Gancedo1,Gancedo2,Granero-Scrobogna} for related  global well-posedness results in Wiener spaces in the critical case, for small enough initial data). 
This result was extended to the 3D case by Gancedo and Lazar~\cite{Gancedo-Lazar-H2} for initial data in 
the critical Sobolev space $\dot{H}^2(\xR^2)$. Eventually, in our companion paper~\cite{AN1}, 
we initiated the study of the Cauchy problem for non-Lipschitz initial data. 

For our subject matter, another fundamental component of the background is that 
the Cauchy problem is not well-posed globally in time:  
there are blow-up results for some large enough data 
by Castro, C\'{o}rdoba, Fefferman, 
Gancedo and L\'opez-Fern\'andez~(\cite{CCFG-ARMA-2013,CCFG-ARMA-2016,CCFGLF-Annals-2012}). 
More precisely, they proved the existence of solutions such that 
at initial time $t=0$ the interface is a graph, at a later time $t_1>0$ 
the interface is no longer a graph and then at a subsequent time $t_2>t_1$, 
the interface is $C^3$ but not $C^4$. 

Our main result in this paper is the following

\begin{theorem}\label{main}
$i)$ For any initial data $f_0$ in $\dot W^{1,\infty}(\xR)\cap H^{\tdm}(\xR)$, there exists a time $T>0$ such that 
the Cauchy problem for the Muskat equation has a unique solution
\begin{equation*}
f\in L^\infty\big([0,T];\dot W^{1,\infty}(\xR)\cap H^{\tdm}(\xR)\big) \cap L^2(0,T;\dot H^2(\xR)).
\end{equation*}
$ii)$ Moreover, there exists a positive constant $\delta$ such that, 
for any initial data $f_0$ in $\dot W^{1,\infty}(\xR)\cap  H^{3/2}(\xR)$ satisfying
\begin{equation*}
\big(1+\lA f_0\rA_{\dot W^{1,\infty}}^4\big)\lA f_0\rA_{\dot H^{\tdm}}\leq \delta,
\end{equation*}
the Cauchy problem for the Muskat equation has a unique global solution
\begin{equation*}
f\in L^\infty\big([0,+\infty);\dot W^{1,\infty}(\xR)\cap H^{\tdm}(\xR)\big)
\cap L^2(0,+\infty;\dot H^2(\xR)).
\end{equation*}
\end{theorem}

Some remarks are in order.

$\bullet$ Let us discuss statement $ii)$ about  the global well-posedness component of this result. This is a $2D$ analogous to 
the recent result by Gancedo and Lazar~\cite{Gancedo-Lazar-H2} for 
the $3D$ problem; it improves on a previous result 
by C\'ordoba and Lazar~\cite{Cordoba-Lazar-H3/2} 
which proves a similar global existence result for 
the $2D$-problem with a similar smallness assumption, but under the extra assumption that the initial 
data belongs to $H^{5/2}(\xR)$. 

$\bullet$ We now come to statement $i)$ about the local well-posedness result for arbitrary initial data. 
This is, in our opinion, the main new result in this paper. 
Since we are working in a critical space, this result is optimal in several directions. 

Firstly, it follows from the results about singularity formation by Castro, C\'{o}rdoba, Fefferman, 
Gancedo and L\'opez-Fern\'andez~(\cite{CCFG-ARMA-2013,CCFG-ARMA-2016,CCFGLF-Annals-2012}) that one cannot solve the 
Cauchy problem for a time $T$ which depends only on 
the norm of $f_0$ in $\dot W^{1,\infty}(\xR)\cap \dot H^{3/2}(\xR)$. 
Otherwise, one would obtain a global existence result for any initial data 
by an immediate scaling argument using~\e{acritical}. 
Notice that this argument 
does not contradict our main result: it means 
instead that the time of existence must 
depend on the initial data itself, and not only on its norm. 

The previous discussion shows that one cannot prove statement $i)$ by using classical Sobolev energy estimates. 
This in turn poses new challenging questions since on the other hand the Muskat equation is a quasi-linear equation. 
To overcome this problem, we will estimate 
the solution for a norm whose definition depends on the initial data. 

$\bullet$ We will also prove a result which elaborates on the previous discussion, stating 
that whenever one controls a bigger norm than the critical one, 
the time of existence is bounded from below on a neighborhood of the initial data. 

To introduce this result, let us fix some notations.

\begin{definition}\label{defi:D}
Given a real number $s\ge 0$ and 
a function $\phi\colon [0,\infty)\to [1,\infty)$ satisfying the following assumptions:
\begin{enumerate}[$({{\rm H}}1)$]
\item\label{H1} $\phi$ is increasing and $\lim\phi(r)=\infty$ when $r$ goes to $+\infty$;
\item\label{H2} there is a positive constant $c_0$ such that $\phi(2r)\leq c_0\phi(r)$ for any $r\geq 0$; 
\item\label{H3} the function $r\mapsto \phi(r)/\log(4+r)$ is decreasing on  $[0,\infty)$.
\end{enumerate}
Then $|D|^{s,\phi}$ denotes the Fourier multiplier with symbol $|\xi|^s\phi(|\xi|)$, so that
\begin{equation*}
\mathcal{F}( |D|^{s,\phi}f)(\xi)=|\xi|^s\phi(|\xi|) \mathcal{F}(f)(\xi).
\end{equation*}
Moreover, we define the space
$$
\mathcal{X}^{s,\phi}(\xR)=\{ f\in \dot{W}^{1,\infty}(\xR)\cap L^2(\xR)\, :\, \D^s \phi(\la D_x\ra)f\in L^2(\xR)\},
$$
equipped with the norm
$$
\lA f\rA_{\mathcal{X}^{s,\phi}}\defn \lA f\rA_{\dot{W}^{1,\infty}}+\lA f\rA_{L^2}
+\left(\int_\xR \la \xi\ra^{2s}(\phi(\la\xi\ra))^2\bla\hat{f}(\xi)\bra^2\dxi\right)^\mez.
$$
\end{definition}
\begin{remark}
The Fourier multiplier $|D|^{s,\phi}$ with 
$\phi(r)=\log(2+r)^a$
was introduced and studied in~\cite{BN18a,BN18b,BN18d} for $s\in [0,1)$ (also see \cite{Ng}). 
\end{remark}
\begin{theorem}\label{T2}
Consider a real number $M_0>0$ and a 
function $\phi$ satisfying assumptions~$(\rm{H}\ref{H1})$--$(\rm{H}\ref{H3})$ in Definiton~\ref{defi:D}. 
Then there exists a time $T_0>0$ such that, 
for any initial data $f_0$ in $\mathcal{X}^{\tdm,\phi}(\xR)$ 
satisfying 
$$
\lA f_0\rA_{\mathcal{X}^{\tdm,\phi}}\le M_0,
$$
the Cauchy problem for the Muskat equation has a unique solution
\begin{equation}
f\in L^\infty\big([0,T_0];\dot W^{1,\infty}(\xR)\cap {H}^{\tdm}(\xR)\big) \cap L^2(0,T_0;\dot H^2(\xR)).
\end{equation}
\end{theorem}
\begin{remark}Statement $i)$ in Theorem~\ref{main} is a consequence of Theorem~\ref{T2}. 
Indeed, it is easily seen that (cf \cite[Lemma~$3.8$]{AN1}), for any $f_0$ in the critical space 
$\dot W^{1,\infty}(\xR)\cap \dot{H}^{\tdm}(\xR)$, one may find a function $\phi$ 
such that $f_0$ belongs to $\mathcal{X}^{\tdm,\phi}(\xR)$ (and satisfying assumptions~$(\rm{H}\ref{H1})$--$(\rm{H}\ref{H3})$ in Definiton~\ref{defi:D}). 
\end{remark}

Theorem~\ref{main} and Theorem~\ref{T2} are proved in the next section.

\subsection*{Acknowledgments} 
\noindent  T.A.\ acknowledges the SingFlows project (grant ANR-18-CE40-0027) 
of the French National Research Agency (ANR).  Q-H.N.\ 
is  supported  by the Shanghai Tech University startup fund and the National Natural Science Foundation of China (12050410257).

The authors would like to thank the referees for their comments, which help to improve the presentation of this article, 
as well as Gustavo Ponce for pointing out a mistake in a preliminary version.
\section{Proof}

\subsection{Regularization}
In order to rigorously justify the computations, 
we want to handle smooth functions (hereafter, a `smooth function' is by definition 
a function that belongs to $C^1([0,T];H^\mu(\xR))$ for any $\mu\in [0,+\infty)$ and some $T>0$). 
To do so, we must regularize the initial data and also consider an approximation of the 
Muskat equation. For our purposes, we further need to consider a regularization of the 
Muskat equation 
which will be compatible with the Sobolev and Lipschitz estimates. 
It turns out that this is a delicate technical problem. 

Our strategy will consist in smoothing the equation in two different ways: 
$i)$ by introducing a cut-off function in the singular integral, removing wave-length shorter than 
some parameter $\eps$ and $ii)$ by 
adding a parabolic term of order $2$ with a small viscosity of size $\la \log(\eps)\ra^{-1}$. 

More precisely, we introduce the following Cauchy problem depending on the parameter $\eps\in (0,1]$: 
\be\label{n2}
\left\{
\begin{aligned}
&\partial_tf-|\log(\varepsilon)|^{-1}\partial_x^2 f
=\frac{1}{\pi}\int_\xR\frac{\partial_x\Delta_\alpha f}{1+\left(\Delta_\alpha f\right)^2}\left(1-\chi\left(\frac{\alpha}{\varepsilon}\right)\right)\dalpha,\\
&f\arrowvert_{t=0}=f_0\star \chi_\eps,
\end{aligned}
\right.
\ee
where $\chi_\eps(x)=\eps^{-1}\chi(x/\eps)$ where 
$\chi$ is a smooth bump function satisfying $0\le \chi\le 1$ and
$$
\chi(y)=\chi(-y),\quad 
\chi(y)=1 \quad\text{for}\quad |y|\le \frac14, \quad \chi(y)=0 \quad\text{for }\la y\ra\ge 2,\quad 
\int_\xR\chi \dy=1.
$$
The equation~\e{n2} does not belong to a general class of parabolic equations. 
However, we will see that it can be studied by standard tools in functional analysis together 
with two estimates for the nonlinearity in the Muskat equation which plays a central role in 
our analysis. 

\begin{proposition}\label{P:initiale}
For any $\eps$ in $(0,1]$ and any initial data $f_0$ 
in $H^{\tdm}(\xR)$, there exists a unique global in time solution $f_\eps$ satisfying
$$
f_\eps\in C^1([0,+\infty);H^\infty(\xR)).
$$
%Moreover, either $T_\eps=+\infty$ or 
%$$
%\limsup_{t\to T_\eps}\lA f_\eps(t)%\rA_{H^{\tdm}}=+\infty.%
%$$
\end{proposition}
We postpone the proof of this proposition to \S\ref{S:end}.

\subsection{An estimate of the Lipschitz norm}

\begin{lemma}\label{L:2.1}
For any real number $\beta_0$ in $(0,1/2)$, there exists a positive constant $C_0\geq 1$ such that, for any $\eps\in (0,1]$ 
and any 
smooth solution $f\in C^1([0,T];H^{\infty}(\xR))$ 
of the Muskat equation \eqref{n2},
\be\label{a1}
\fract \lA f(t)\rA_{\dot W^{1,\infty}}\le C_0 \lA f(t)\rA_{\dot H^2}^2+C_0 \varepsilon^{\beta_0}\lA f(t)\rA_{\dot C^{2,\beta_0}},
\ee
where
$$
\lA u\rA_{\dot C^{2,\beta_0}}=\lA \partial_{xx}u\rA_{C^{0,\beta_0}}
=\sup_{\substack{ x,y\in \xR \\ x\neq y}} \frac{\la (\partial_{xx}u)(x)-(\partial_{xx}u)(y)\ra}{\la x-y\ra^{\beta_0}}\cdot
$$
\end{lemma}
\begin{proof}
The proof is partially based on arguments from~\cite{CG-CMP2,Cameron,Gancedo-Lazar-H2}. 
Firstly, it follows from the proof of \cite[Lemma~$5.1$]{CG-CMP2} that
\begin{align*}
\partial_{x}\frac{1}{\pi}\int_\xR\frac{\partial_x\Delta_\alpha f(x)}{1+\left(\Delta_\alpha f(x)\right)^2}\dalpha&=\frac{\partial_x^2f(t,x)}{2\pi}
\int\left(\frac{1}{1+(\Delta_\alpha f(t,x))^2}-\frac{1}{1+(\Delta_{-\alpha} f(t,x))^2}\right)\frac{\dalpha}{\alpha}\\
&\quad-\frac{2}{\pi}
\int\frac{\partial_x f(t,x)-\Delta_\alpha f(t,x)}{\alpha^2}
\frac{1+\partial_x f(t,x)\Delta_\alpha f(t,x)}{1+(\Delta_\alpha f(t,x))^2} \dalpha.
\end{align*}
Moreover, 
\be\label{X3}
\begin{aligned}
&\left|\partial_{x}\left(\frac{1}{\pi}\int_\xR\frac{\partial_x\Delta_\alpha f}{1+\left(\Delta_\alpha f\right)^2}\chi\left(\frac{|\alpha|}{\varepsilon}\right)\dalpha\right)\right|\\
&\qquad\qquad\lesssim \int_{|\alpha|\leq 2 \varepsilon}\left(|\Delta_\alpha f_{xx}|
+|\Delta_\alpha f_{x}|^2\right)\dalpha \\&
\qquad\qquad\lesssim \int_\xR |\Delta_\alpha f_{x}|^2\dalpha+\varepsilon^{\beta_0}\lA f_{xx}\rA_{\dot C^{0,\beta_0}},
\end{aligned}
\ee
where we used the notations $f_x=\partial_x f$ and $f_{xx}=\partial_{xx}f$. 
Thus, for any $t$ and any $x$, we have
\be\label{n10}
\begin{aligned}
&(\partial_x \partial_t f)(t,x) -|\log(\varepsilon)|^{-1}\partial_{x}^2 f_x(t,x)
\\
&\qquad\qquad\leq \frac{\partial_x^2f(t,x)}{2\pi}
\int\left(\frac{1}{1+(\Delta_\alpha f(t,x))^2}-\frac{1}{1+(\Delta_{-\alpha} f(t,x))^2}\right)\frac{\dalpha}{\alpha}\\
&\qquad\qquad\quad-\frac{2}{\pi}
\int\frac{\partial_x f(t,x)-\Delta_\alpha f(t,x)}{\alpha^2}
\frac{1+\partial_x f(t,x)\Delta_\alpha f(t,x)}{1+(\Delta_\alpha f(t,x))^2} \dalpha\\
&\qquad\qquad\quad +C  \int |\Delta_\alpha f_{x}(t,x)|^2\dalpha+C\varepsilon^{\beta_0}\lA f_{xx}(t)\rA_{\dot C^{0,\beta_0}}.
\end{aligned}
\ee
Consider the function $\varphi(t)=\lA \partial_{x}f(t)\rA_{L^\infty}$ and 
a function $ t\mapsto x_t$ 
such that
$$
\lA \partial_{x}f(t)\rA_{L^\infty}=(\partial_{x} f)(t,x_t).
$$
Then $(\partial_{x}^2f)(t,x_t)=0$ and $-(\partial_{xx} f_x)(t,x_t)\geq 0$. So, 
it follows 
from \e{n10} that
\begin{align*}
\dot{\varphi}(t)&\leq -\frac{2}{\pi}\int
\frac{\partial_x f(t,x_t)-\Delta_\alpha f(t,x_t)}{\alpha^2}\dalpha\\&
\quad-\frac{2}{\pi}\int
\frac{(\partial_x f(t,x_t)-\Delta_\alpha f(t,x_t))^2}{\alpha^2}
\frac{\Delta_\alpha f(t,x_t)}{1+(\Delta_\alpha f(t,x_t))^2} \dalpha\\& \quad+C  \int |\Delta_\alpha f_{x}(t,x_t)|^2\dalpha+C\varepsilon^{\beta_0}\lA f_{xx}(t)\rA_{\dot C^{0,\beta_0}}.
\end{align*}
As already observed in~\cite{CG-CMP2} (see also~\cite{Cameron,Gancedo-Lazar-H2}), 
the first term in the right-hand side has a sign 
since $\partial_x f(t,x_t)\ge 
\Delta_\alpha f(t,x_t)$ for any $\alpha$. 
It follows that
\begin{align*}
\dot{\varphi}(t)&\leq \frac{1}{\pi}\int \frac{(\partial_x f(t,x_t)
-\Delta_\alpha f(t,x_t))^2}{\alpha^2} \dalpha+C  \int |\Delta_\alpha f_{x}(t,x_t)|^2\dalpha\\
&\quad+C\varepsilon^{\beta_0}\lA f_{xx}(t)\rA_{\dot C^{0,\beta_0}}.
\end{align*}
We now apply Hardy's inequality to infer that
$$
\int \frac{(\partial_x f(t,x_t)
-\Delta_\alpha f(t,x_t))^2}{\alpha^2} \dalpha\les \int |\Delta_\alpha f_{x}(t,x_t)|^2\dalpha.
$$
Consequently, we end up with
\begin{align*}
\dot{\varphi}(t)\lesssim 
\int \lA\Delta_\alpha f_{x}(t)\rA_{L^\infty}^2\dalpha +\varepsilon^{\beta_0}\lA f_{xx}(t)\rA_{\dot C^{0,\beta_0}}.
\end{align*}
Introducing the difference 
operator $\delta_\alpha g(x)=g(x)-g(x-\alpha)$, the previous inequality 
is better formulated as follows:
\begin{align*}
\dot{\varphi}(t)\lesssim 
\int \lA \delta_\alpha (\partial_xf)(t)\rA_{L^\infty}^2
 \frac{\dalpha}{|\alpha|^{1+\mez 2}}+\varepsilon^{\beta_0}\lA f_{xx}(t)\rA_{\dot C^{0,\beta_0}}.
\end{align*}
Now the right-hand side is equivalent to 
the following homogeneous Besov norm:
$\lA \partial_x f(t)\rA_{\dot{B}^{\mez}_{\infty,2}}^2$ 
(see~\cite{Triebel-TFS,Triebel1988} or Section~$2$ in \cite{AN1}). 
Then it follows from Sobolev embeddings that
$$
\dot{\varphi}(t)\lesssim \lA f(t)\rA_{\dot H^2}^2+\varepsilon^{\beta_0}\lA f_{xx}(t)\rA_{\dot C^{0,\beta_0}}
$$
which is the wanted result.
\end{proof}

\subsection{Sobolev estimates}
In this paragraph we recall a generalized Sobolev 
energy estimate proved in our companion paper~\cite{AN1}. 
By generalized Sobolev energy estimate, we mean that, 
instead of estimating the $L^\infty_t(L^2_x)$-norm 
of $(-\Delta)^s f$, we shall estimate 
the $L^\infty_t(L^2_x)$-norm of $\D^{s,\phi}f$ for some function $\phi$ 
satisfying the assumptions in Definition~\ref{defi:D}. 

There two technical results 
that we will borrow from~\cite{AN1}. The first result, which is Lemma~$3.4$ in \cite{AN1}, 
gives an energy estimate.

\begin{lemma}\label{L:3.4}
There exists a positive constant $C$ such that, for any $T>0$ and 
any smooth solution $f\in C^{1}([0,T];H^{\infty}(\xR))$ to~\e{n1}, there holds 
\begin{multline}\label{Z21}
\fract \blA \D^{\tdm,\phi}f\brA_{L^2}^2
+ \int_\xR \frac{\bla\D^{2,\phi}f\bra^2}{1+(\partial_x f)^2}\dx+|\log(\varepsilon)|^{-1}\int_\xR \bla\D^{\frac{5}{2},\phi}f\bra^2\dx\\
\le C Q(f) \blA\D^{2,\phi}f\brA_{L^2},
\end{multline}
where
\begin{align*}
Q(f)&= \left(\lA f\rA_{\dot H^2}+\lA f\rA_{\dot H^{\frac{7}{4}}}^2\right) 
\blA\D^{\tdm,\phi}f\brA_{L^2}
+\blA\D^{\frac74,\phi}f\brA_{L^2}
\lA f\rA_{{H}^{\frac74}}\\
&\quad+\left(\lA f\rA_{H^{\frac{19}{12}}}^{3/2}+\lA f\rA_{\dot H^{\frac74}}^{1/2}\right) \blA\D^{\frac{7}{4},\phi^{2}}f\brA^{1/2}_{L^2}
\lA f\rA_{\dot H^{\frac74}}.
\end{align*} 
\end{lemma}
\begin{remark}
Some explanations are in order since 
the reader may notice several modifications compared to our paper~\cite{AN1}. Firstly, 
in \cite{AN1} we considered 
a function $\phi$ whose definition depends on an extra function $\kappa$. 
Here we ignore this point 
since it is irrelevant for the present analysis. Indeed, 
the functions $\phi$ and $\kappa$ are 
shown in \cite{AN1} to be equivalent (such that $c\kappa(\lam)\le \phi(\lam)\le C \kappa(\lam)$), and the distinction between 
them served only to organize the proof. 
Secondly, in~\cite{AN1} we also assume that $\phi(r)$ is bounded from below by $(\log(4+ r))^a$ for some $a\ge 0$. 
Here we will use that this property holds with~$a=0$. 
Once the previous clarifications have been done, 
it remains to explain that in \cite{AN1} we consider the equation~\e{n1} 
while here we work with~\e{n2}. The elliptic term $(-\partial_{x}^2)$ is trivial to handle since 
in \cite{AN1} we only applied an $L^2$-energy estimate and since the latter operator is positive. 
Eventually, the cut-off function $(1-\chi(\alpha/\eps))$ is also harmless 
in the various computations used to prove Lemma~$3.4$ in~\cite{AN1}. 
\end{remark}

Secondly, we recall two interpolation inequalities from \cite[Lemma~$3.5$]{AN1}. 
Hereafter, we use the notations
\be\label{n67}
\begin{aligned}
A_\phi(t)&=\blA \D^{\tdm,\phi}f(t)\brA_{L^2}^2,\\
B_\phi(t)&=\blA \D^{2,\phi}f(t)\brA_{L^2}^2,
 \\
P_\phi(t)&=\blA \D^{\frac{5}{2},\phi}f(t)\brA_{L^2}^2,
\end{aligned}
\ee
and 
$$
\mu_\phi(t)=\left(\phi\left(\frac{B(t)}{A(t)}\right)\right)^{-1} .
$$
\begin{lemma}\label{L:3.5}
Consider a real number $7/4\le s\leq 2$. Then, 
there exists a positive constant $C$ such that, for any $T>0$, 
any smooth solution $f\in C^{1}([0,T];H^{\infty}(\xR))$ to~\e{n2} and any $t\in [0,T]$,
\begin{align}
&\lA f(t)\rA_{\dot H^s}\le C\mu_\phi(t) A_\phi(t)^{2-s}  B_\phi(t)^{s-\tdm},\label{Z20'}\\
&\blA \D^{\frac{7}{4},\phi^{2}}f(t)\brA_{L^2}\leq C \mu_\phi(t)A_\phi(t)^{\frac{1}{4}}B_\phi(t)^{\frac{1}{4}}.\label{n110}
\end{align}
\end{lemma}

From these two lemmas, we get at once the following

\begin{proposition}\label{P:3.3}
There exist two positive constants $C_1$ and $C_2$ such that, 
for any $T>0$ and any smooth solution $f\in C^1([0,T];H^\infty(\xR))$ of the Muskat equation~\e{n2},
\begin{multline}\label{a2}
\fract A_\phi(t)+C_1\frac{B_\phi(t)}{1+\lA f_x(t)\rA_{L^\infty}^2}+|\log(\varepsilon)|^{-1} P_\phi(t)\\
\leq C_2 \left( \sqrt{A_\phi(t)}+A_\phi(t) \right)\mu_\phi(t) B_\phi(t).
\end{multline}
\end{proposition}

We will also need an estimate for the $L^2$-norm. 
%usual non-homogeneous Sobolev norms. 
%\begin{proposition}\label{P:3.3'}
%There exist two positive constants $C_1$ and $C_2$ such that, 
%for any $T>0$ and any smooth solution $f\in C^1([0,T];H^\infty(\xR))$ of the Muskat equation~\e{n2},
%\begin{multline}\label{a2'}
%\fract A_\phi(t)+C_1\frac{B_\phi(t)}{1+\lA f_x(t)\rA_{L^\infty}^2}+|\log(\varepsilon)|^{-1} P_\phi(t)\\
%\leq C_2 \left( \sqrt{A_\phi(t)}+A_\phi(t) \right)\mu_\phi(t) B_\phi(t).
%\end{multline}
%\end{proposition}
%\begin{proof}
%In view of \e{a2}, it suffices to estimate the $L^2$-norm of the solution. 
%This follows by an elementary energy estimate, multiplying the equation by $f$ and using 
%the estimate~\e{f4} proved below for the $L^2$-norm of the right-hand side.
%\end{proof}
\begin{lemma}\label{L:L2}
There holds
$$
\frac{1}{2}\fract \lA f(t)\rA_{L^2}^2\le C \varepsilon^{\mez}\lA f\rA_{\dot{H}^{\tdm}}\lA f\rA_{L^2}.
$$
In particular,
\begin{equation}\label{X3deux}
 \lA f(t)\rA_{L^2}\le  \lA f_0\rA_{L^2} +C\eps^{\mez}\int_0^t\lA f(\tau)\rA_{\dot{H}^{\tdm}}\dtau.
\end{equation}
\end{lemma}
\begin{proof}
Set
\be\label{f10}
R_\eps(f)=-\frac{1}{\pi}\int_\xR\frac{\partial_x\Delta_\alpha f}{1+\left(\Delta_\alpha f\right)^2}
\chi\left(\frac{\alpha}{\varepsilon}\right)\dalpha.
\ee
We multiply the equation by $f$ to obtain
$$
\frac{1}{2}\fract \lA f(t)\rA_{L^2}^2\le
\frac{1}{\pi}\bigg\langle \int_\xR\frac{\partial_x\Delta_\alpha f}{1+\left(\Delta_\alpha f\right)^2}\dalpha , f\bigg\rangle
+\langle R_\eps(f),f\rangle.
$$
Now, by \cite[Section 2]{CCGRPS-JEMS2013}, the first term in the 
right-hand side has a sign. Indeed: 
\begin{multline*}
\int_\xR\left[\int_\xR\frac{\partial_x\Delta_\alpha f}{1+\left(\Delta_\alpha f\right)^2}\dalpha \right]f(x)\dx
\\
=-\iint_{\xR^2}\log\left[\sqrt{1+\frac{(f(t,x)-f(t,x-\alpha))^2}{\alpha^2}}\right]\dx\dalpha.
\end{multline*}
It remains to estimate $R_\eps(f)$. To do so, we use the estimate~\e{n15AN1} to get
\be\label{f9}
\begin{aligned}
\lA R_\eps(f)\rA_{L^2}
&\lesssim  \int_{|\alpha|\leq 2\varepsilon}\lA\Delta_\alpha f_{x}\rA_{L^2}\dalpha \\
&  \lesssim \varepsilon^{\mez}\left(\int_{\xR}\lA\Delta_\alpha f_{x}\rA_{L^2}^2\dalpha\right)^\mez
  \lesssim \varepsilon^{\mez}\lA f\rA_{\dot{H}^{\tdm}},
\end{aligned}
\ee
which completes the proof.
\end{proof}
\subsection{Estimate of the H\"older norm}
To exploit the Sobolev energy estimate~\e{a2}, the main difficulty is to estimate from above the factor $1+\lA f_x(t)\rA_{L^\infty}^2$. 
This is where we will apply Lemma~\ref{L:2.1}. This in turn requires to 
estimate the H\"older norm $\lA \cdot\rA_{\dot{C}^{2,\beta_0}}$ of $f$. 
This is the purpose of the following result.

We will prove an estimate valid on arbitrary large time scale, which will be used later 
to prove a global existence result. 
\begin{proposition}\label{P:2.6}
For any $0<\betathree<1/2$, 
there exist two positive constant $\eps_0$ and $c_0$  
such that, for any $\eps\in (0,\eps_0]$, any 
smooth solution $f\in C^1([0,T];H^{\infty}(\xR))$ 
of the Muskat equation~\eqref{n2}, and any time $t\leq \min\{\varepsilon^{-c_0},T\}$, 
there holds
\begin{align*}
&\varepsilon^{\betathree}\int_{0}^{t}\lA f(\tau)\rA_{C^{2,\betathree}}d\tau 
\le  \varepsilon^{\frac{\betathree}{2}}\lA f_0\rA_{\dot H^{\frac{3}{2}}} \\
&~~~+ \varepsilon^{\frac{\betathree}{2}}\left(1+\sup_{s\in [0,t]}\lA f(s)\rA_{H^{\frac{3}{2}}}\right)^2
\log\left(2+\int_{0}^{t} \lA f(s)\rA_{\dot H^{2}}^2\ds\right)^{\mez}
\left(\int_{0}^{t} \lA f(s)\rA_{\dot H^{2}}^2\ds\right)^{\mez}.\nonumber
\end{align*}
\end{proposition}

\begin{proof}
The classical Sobolev embeddings implies that
$$
\lA f(t)\rA_{\dot{C}^{2,\betathree}}
\lesssim \lA f(t)\rA_{\dot H^{\frac{5}{2}+\betathree}}.
$$
To estimate the latter Sobolev norm, the key point will be to apply the following interpolation inequality.

\begin{lemma}
Consider three real numbers
$$
\gamma>0,\quad \beta_1>0\quad\text{and}\quad 0<\beta_2<2.
$$
Then, there exists a constant $C$ such that, for any function $g=g(t,x)$, 
\begin{equation}\label{X2}
\begin{aligned}
\lA g(t)\rA_{\dot H^\gamma}
&\lesssim \frac{1}{(\nu t)^{\frac{\beta_1}{2}}}
\lA g(0)\rA_{\dot H^{\gamma-\beta_1}}\\
&\quad+\int_0^{t}  \frac{1}{(\nu (t-s))^{\frac{\beta_2}{2}}}\blA 
\partial_tg(s)-\nu\partial_{xx} g(s)\brA_{\dot H^{\gamma-\beta_2}}\ds.
\end{aligned}
\end{equation}
\end{lemma}
\begin{proof}
Set $G\defn\partial_tg-\nu\partial_{xx}  g$. Then, one has, 
\begin{equation*}
\hat g(t,\xi)=e^{-\nu t|\xi|^{2}}\hat g(0,\xi)+\int_{0}^{t}e^{-\nu (t-s)|\xi|^{2}}\hat G(s,\xi)\ds.
\end{equation*}
The desired results then follows from Minkowski's inequality. 
\end{proof}

Now, apply \eqref{X2} with
$$
\gamma=\frac{5}{2}+\betathree,
\quad\beta_1=1+\betathree, \quad\beta_2=\frac{3}{2}+\betathree,\quad \nu=|\log(\varepsilon)|^{-1},
$$
to get
\begin{align}
\lA f(t)\rA_{\dot C^{2,\betathree}}&\les \lA f(t)\rA_{\dot H^{\frac{5}{2}+\betathree}}\nonumber\\
&\lesssim |\log(\varepsilon)|^{\frac{1+\betathree}{2}} t^{-\frac{1+\betathree}{2}}||f_0||_{\dot H^{\frac{3}{2}}}\nonumber\\
&\quad+\int_0^{t}   |\log(\varepsilon)|^{\frac{\frac{3}{2}+\betathree}{2}}
(t-s)^{-\frac{\frac{3}{2}+\betathree}{2}} \blA\partial_tf-|\log(\varepsilon)|^{-1}\partial_x^2f\brA_{\dot H^1}\ds.\label{Y1}
\end{align}

It remains to estimate the $\dot{H}^1$-norm of $\partial_tf-|\log(\varepsilon)|^{-1}\partial_{xx} f$. 
In view of the equation~\e{n2}, this is equivalent to bound the $\dot{H}^1$-norm of 
$$
\frac{1}{\pi}\int_\xR\frac{\partial_x\Delta_\alpha f}{1+\left(\Delta_\alpha f\right)^2}\left(1-\chi\left(\frac{\alpha}{\varepsilon}\right)\right)\dalpha.
$$
We will split the latter term into two pieces and estimate them separately. 

Firstly, directly from~\eqref{X3} and Minkowski's inequality, we obtain that
\begin{align*}
\lA \frac{1}{\pi}\int_\xR\frac{\partial_x\Delta_\alpha f}{1+\left(\Delta_\alpha f\right)^2}
\chi\left(\frac{\alpha}{\varepsilon}\right)\dalpha\rA_{\dot H^1}
&\lesssim  \int_{|\alpha|\leq 2\varepsilon}\left(\lA\Delta_\alpha f_{xx}\rA_{L^2}
+\lA\Delta_\alpha f_{x}\rA_{L^4}^2\right)\dalpha\\
&  \lesssim \varepsilon^{\mez+\beta}\left(\int_{\xR}\lA\Delta_\alpha f_{xx}\rA_{L^2}^2|\alpha|^{-2\beta}\dalpha\right)^\mez+
\int_{\xR}\lA\Delta_\alpha f_{x}\rA_{L^4}^2\dalpha.
\end{align*}
Now we use the following inequality: 
\be\label{n15AN1}
\iint_{\xR^2}  \bla \Delta_\alpha \tilde{f}\bra^2 |\alpha|^{-2\beta}\dalpha\dx\sim 
\blA \tilde{f}\brA_{\dot{H}^{\mez+\beta}}^2.
\ee
Indeed,
$$
\iint_{\xR^2}  \bla \Delta_\alpha \tilde{f}\bra^2 |\alpha|^{-2\beta}\dalpha\dx=
\iint_{\xR^2}  \left[\frac{\bla \tilde{f}(x)-\tilde{f}(x-\alpha)\bra}{\la \alpha\ra^{1/2+\beta}}\right]^2\frac{\dalpha}{\la\alpha\ra}\dx
\sim \blA \tilde{f}\brA_{\dot{H}^{\mez+\beta}}^2.
$$
Similarly, using Sobolev embedding in Besov's spaces, we get
$$
\int_{\xR}\lA\Delta_\alpha f_{x}\rA_{L^4}^2\dalpha\les \lA f\rA_{\dot H^{\frac{7}{4}}}^2.
$$
It follows that
\be\label{f1}
\lA \frac{1}{\pi}\int_\xR\frac{\partial_x\Delta_\alpha f}{1+\left(\Delta_\alpha f\right)^2}
\chi\left(\frac{\alpha}{\varepsilon}\right)\dalpha\rA_{\dot H^1}
\lesssim \varepsilon^{\mez+\beta}\lA f\rA_{\dot{H}^{\frac{5}{2}+\beta}}+ \lA f\rA_{\dot H^{2}}  \lA f\rA_{\dot H^{\frac{3}{2}}},
\ee
where we used an interpolation inequality in Sobolev spaces. 
On the other hand, it follows from the estimate~\eqref{n33} below that,
\be\label{f2}
\begin{aligned}
\lA \int_\xR\frac{\partial_x\Delta_\alpha f}{1+\left(\Delta_\alpha f\right)^2}\dalpha\rA_{\dot H^1}
&\les  \lA \mathcal{T}(f)f\rA_{\dot H^1}\\
&\les \left(
1+\lA f\rA_{H^{\frac32}}\right)^2
\log\left(2+\lA f\rA_{\dot H^2}^2\right)^{\mez} \lA f\rA_{\dot H^{2}}.
\end{aligned}
\ee
By gathering the two previous estimates, we conclude that
\begin{multline*}
\lA\partial_tf-|\log(\varepsilon)|^{-1}\partial_x^2f\rA_{\dot H^1}\\
\lesssim \varepsilon^{\frac{1}{2}+\beta}\lA f\rA_{H^{\frac{5}{2}+\beta}}+
\left(1+\lA f\rA_{H^{\frac{3}{2}}}\right)^2\log\left(2+\lA f\rA_{\dot H^2}^2\right)^{\mez}  \lA f\rA_{\dot H^{2}}.
\end{multline*}
Set
$$
b=\frac{\frac{3}{2}+\betathree}{2}\cdot
$$
By reporting this bound in~\e{Y1}, we find that
\begin{align*}
&\lA f(t)\rA_{\dot H^{\frac{5}{2}+\betathree}}\lesssim |\log(\varepsilon)|^{\frac{1+\betathree}{2}} t^{-\frac{1+\betathree}{2 }}\lA f_0\rA_{\dot H^{\frac{3}{2}}}\\
&\quad + \varepsilon^{\frac{1}{2}+\beta}|\log(\varepsilon)|^{b}\int_0^{t}    (t-s)^{-b} \lA f(s)\rA_{\dot H^{\frac{5}{2}+\beta}}\ds\\
&\quad+|\log(\varepsilon)|^{b}\int_0^{t}    (t-s)^{-b} 
\left(1+\lA f(s)\rA_{H^{\frac{3}{2}}}\right)^2\log\left(2+\lA f(s)\rA_{\dot H^2}^2\right)^{\mez}  \lA f(s)\rA_{\dot H^{2}}\ds.
\end{align*}
So, 
\begin{align*}
&\int_{0}^{t}\lA f(\tau )\rA_{\dot H^{\frac{5}{2}+\betathree}}\dtau \lesssim \la\log(\varepsilon)\ra^{\frac{1+\betathree}{2 }} t^{1-\frac{1+\betathree}{2 }}\lA f_0\rA_{\dot H^{\frac{3}{2}}}\\
&\quad+ \varepsilon^{\frac{1}{2}+\beta}|\log(\varepsilon)|^{b} t^{1-b} \int_0^{t}\lA f(s)\rA_{\dot H^{\frac{5}{2}+\beta}}\ds\\
&\quad+|\log(\varepsilon)|^{b}
t^{1-b}\int_0^{t}\left(1+\lA f(s)\rA_{H^{\frac{3}{2}}}\right)^2\log\left(2+\lA f(s)\rA_{\dot H^2}^2\right)^{\mez}  \lA f(s)\rA_{\dot H^{2}}\ds.
\end{align*}

As a result, there exists $c_0>0$ and $\eps_0\le 1$ such that, if $t\leq \varepsilon^{-c_0}$ and $\varepsilon\le \eps_0$, 
\begin{align*}
\int_{0}^{t}\lA f(\tau )\rA_{\dot H^{\frac{5}{2}+\betathree}}\dtau &\le \varepsilon^{-\frac{\betathree}{2}}||f_0||_{\dot H^{\frac{3}{2}}} \\
&\quad +|\log(\varepsilon)|^{b}  t^{1-b} \mathcal{K}(t)\int_0^{t}    \log\left(2+\lA f(s)\rA_{\dot H^2}^2\right)^{\mez}  \lA f(s)\rA_{\dot H^{2}}\ds,
\end{align*}
where
$$
\mathcal{K}(t)=\sup_{s\in [0,t]}\left(1+\lA f(s)\rA_{H^{\frac{3}{2}}}\right)^2.
$$
Now observe that
\begin{multline*}
\int_0^{t}    \log\left(2+\lA f(s)\rA_{\dot H^2}^2\right)^{\mez}  \lA f(s)\rA_{\dot H^{2}}\ds\\
\le (t+1)^{\frac{1}{2}}    \log\left(2+ \int_{0}^{t} \lA f(s)\rA_{\dot H^{2}}^2ds\right)^{\mez}
\left( \int_{0}^{t} \lA f(s)\rA_{\dot H^{2}}^2ds\right)^{\frac{1}{2}}.
\end{multline*}
Therefore, up to modifying the values of $c_0>0$ and $\eps_0$, we see that, 
for $t\leq \varepsilon^{-c_0}$ and $\varepsilon\le \eps_0$, we have
\begin{align*}
&\varepsilon^{\betathree}\int_{0}^{t}\lA f(\tau)\rA_{\dot C^{2,\betathree}}\dtau
\lesssim  \varepsilon^{\frac{\betathree}{2}}\lA f_0\rA_{\dot H^{\frac{3}{2}}} \\&~~~+ \varepsilon^{\frac{\betathree}{2}}\mathcal{K}(t)\log\left(2+ \int_{0}^{t} \lA f(s)\rA_{\dot H^{2}}^2\ds\right)^{\mez} \left( \int_{0}^{t} \lA f(s)\rA_{\dot H^{2}}^2\ds\right)^{\frac{1}{2}}.
\end{align*}
This completes the proof.
\end{proof}

\subsection{Global in time estimates, under a smallness assumption}

\begin{proposition}
Let $T>0$ and consider a 
smooth solution $f\in C^1([0,T],H^{\infty}(\xR))$ of the Muskat equation~\e{n2}. 
Set
$$
K=1+16 \left(\frac{C_2}{C_1}\right)^2
$$
and assume that
\be\label{a29}
2\left(K+\frac{C_0}{C_1}\right)^\mez\left(2+\lA \partial_x f_0\rA_{L^\infty}\right)^2\lA f_0\rA_{\dot {H}^{\tdm}}\le 1,  
\ee
where the constants $C_0,C_1,C_2$ are as defined in the statements of 
Lemma~\ref{L:2.1} and Proposition~\ref{P:3.3}. 
Then there exists $\eps_0$ depending only on $C_0,C_1,C_2$ and $||f_0||_{L^2}$ such that, if $\eps\le \eps_0$, then 
\be\label{a30}
\sup_{0\le \tau\le T}\lA f(\tau)\rA_{{H}^{\tdm}}
\le \frac{1}{\sqrt{K}\big(2+\lA \partial_xf_0\rA_{L^\infty}\big)^2}\quad\text{and}
\quad
\int_{0}^{T}\lA f(\tau)\rA_{\dot{H}^2}^2\dtau\leq \frac{1}{C_0} \cdot
\ee
\end{proposition}
\begin{proof}
We apply the previous {\em a priori\/} estimate~\e{a2} 
in the simplest case where~$\phi=1$. With this choice, the quantities $A_\phi$ and $B_\phi$ defined by~\e{n67} simplify to
\be\label{n167}
\begin{aligned}
A(t)&=\blA \D^{\tdm}f(t)\brA_{L^2}^2, \\
B(t)&=\blA \D^{2}f(t)\brA_{L^2}^2=\lA f(t)\rA_{\dot{H}^2}^2.
\end{aligned}
\ee

Introduce the set
$$
I=\left\{ t\in [0,T]\,;\, \int_{0}^{t}B(\tau)\dtau\leq \frac{2}{3C_0} 
\text{ and } \sup_{0\le \tau\le t}A(\tau)\le \frac{1}{K\big(2+\lA \partial_xf_0\rA_{L^\infty}\big)^4}\right\}.
$$
We want to prove that $I=[0,T]$. Since $0$ belongs to $I$ by assumption on the initial data, and since $I$ is closed, it suffices to prove that $I$ is open. 
To do so, we consider a time $t^*\in [0,T)$ which belongs to $I$. 
Our goal is to prove that
$$
\int_{0}^{t^*}B(\tau)\dtau\leq \frac{1}{2C_0} 
\text{ and } \sup_{0\le \tau\le t^*}A(\tau)\le \frac{1}{4K\big(2+\lA \partial_xf_0\rA_{L^\infty}\big)^4}.
$$
This will imply at once that $t^*$ belongs to the interior of $I$.

Since $\mu(t)=1$ for $\phi\equiv 1$, the estimate~\e{a2} implies that there are two positives 
constants $C_1,C_2$ such that
\be\label{a4}
\fract A(t)+C_1\frac{B(t)}{1+\lA \partial_x f(t)\rA_{L^\infty}^2}\leq C_2\Big(A(t)+\sqrt{A(t)}\Big)B(t).
\ee
By combining Proposition~\ref{P:2.6} with Lemma~\ref{L:2.1}, we get, for any $t$,
\begin{align*}
&	\lA \partial_xf(t)\rA_{L^\infty}-	\lA \partial_xf_0\rA_{L^\infty}\leq C_0\int_{0}^{t}B(\tau) \dtau+C_0 \varepsilon^{\frac{\betathree}{2}}\lA f_0\rA_{\dot H^{\frac{3}{2}}}\\&+
C_0	\varepsilon^{\frac{\betathree}{2}}\left[\sup_{s\in [0,t]}\left(1+\lA f(s)\rA_{H^{\frac{3}{2}}}\right)^2\right]\log\left(2+\int_{0}^{t} B(\tau)\dtau\right)^{\mez}
\left( \int_{0}^{t} B(\tau)\dtau\right)^{\frac{1}{2}}.
\end{align*}
By \eqref{X3deux}, 
\begin{equation*}
	\sup_{s\in [0,t]}\lA f(s)\rA_{L^{2}}\le  \lA f_0\rA_{L^2}+C\eps^{\mez}t \sup_{s\in [0,t]}\lA f(s)\rA_{\dot H^{\frac{3}{2}}}.
\end{equation*}
This implies 
\begin{equation}\label{f11}
 \sup_{s\in [0,t]}\lA f(s)\rA_{H^{\frac{3}{2}}}\le  \lA f_0\rA_{L^2}+(1  +C\eps^{\mez}t) \sup_{s\in [0,t]}\lA f(s)\rA_{\dot H^{\frac{3}{2}}}.
\end{equation}

If $t\le t^*$, then the bound on the integral of $B$,    $ \sup_{0\le \tau\le t^*}||f||_{\dot H^{\frac{3}{2}}}\le1$ and \eqref{f11}  imply  that 
\begin{align*}
	&	\lA \partial_xf(t)\rA_{L^\infty}-	\lA \partial_xf_0\rA_{L^\infty}\leq \frac{1}{2}+C_0 \varepsilon^{\frac{\betathree}{2}}+
	C_0	\varepsilon^{\frac{\betathree}{2}}\left(1+ \lA f_0\rA_{L^2}+(1  +C\eps^{\mez}t^*)\right)^2\log\left(3\right)^{\mez}.
\end{align*}
 For $\eps$ small enough, we conclude that
$$
\lA \partial_xf(t)\rA_{L^\infty}\le \frac{2}{3}+\lA \partial_xf_0\rA_{L^\infty}.
$$
On the other hand, if $t^*\in I$, then for any $t\le t^*$ 
we have
\begin{equation*}
A(t)+\sqrt{A(t)}\le 2\sqrt{A(t)}\le \frac{2}{\sqrt{K}\big(2+\lA \partial_xf_0\rA_{L^\infty}\big)^2}\cdot
\end{equation*}
Consequently, for any $t\leq t^\star$, \eqref{a4} gives
$$
\fract A(t)+C_1\frac{B(t)}{\big(2+\lA \partial_{x}f_0\rA_{L^\infty}\big)^2}
\leq \frac{2C_2}{\sqrt{K}\big(2+\lA \partial_xf_0\rA_{L^\infty}\big)^2}B(t).
$$
By definition of $K$, we have
$$
K\ge \frac{16C_2^2}{C_1^2},
$$
so, 
for any $t\leq t^\star$,
\begin{equation}\label{a20}
\fract A(t)+\frac{C_1}{2}\frac{B(t)}{\big(2
+\lA\partial_{x}f_0\rA_{L^\infty}\big)^2}\leq 0.
\end{equation}
Integrate this on the time interval $[0,t^*]$, to infer that 
\begin{equation*}
\sup_{t\in [0,t^*]}A(t)+\frac{C_1}{2\big(2
+\lA\partial_{x}f_0\rA_{L^\infty}\big)^2}
\int_{0}^{t^*}B(t)\dt\leq A(0).
\end{equation*}
Using the smallness assumption~\e{a29}, the previous inequality~\e{a20} implies at once that
\begin{align*}
&\sup_{t\in [0,t^*]}A(t)\le A(0)\le \frac{1}{4K\big(2+\lA \partial_xf_0\rA_{L^\infty}\big)^4},\\
&\int_{0}^{t^*}B(t)\dt 
\leq \frac{2\big(2+\lA\partial_{x}f_0\rA_{L^\infty}\big)^2}{C_1}A(0)\le \frac{1}{2C_0}.
\end{align*}
These are the wanted bootstrap inequalities. As explained above, by connexity, this proves that $I=[0,T]$, 
which implies the desired results in~\e{a30}.
\end{proof}

\subsection{A priori estimates locally in time, for arbitrary initial data}
\begin{proposition}
Consider $\phi$ satisfying assumptions~$(\rm{H}\ref{H1})$--$(\rm{H}\ref{H3})$ in Definiton~\ref{defi:D}. 
Let $T>0$ and consider a 
smooth solution $f\in C^1([0,T],H^{\infty}(\xR))$ of the Muskat equation~\e{n2}. 
For any $M_0>0$ there exists $\eps_0>0$ and $T_0>0$ such that the following properties holds. 
If $\eps\in (0,\eps_0]$ and 
$$
\blA \D^{\tdm,\phi}f(0)\brA_{L^2}^2\le M_0,
$$
then, with $T^*=\min\{T,T_0\}$, there holds
$$
\sup_{t\in [0,T^*]}A_\phi(t)\le 5M_0,\quad \int_0^{T^*}\mu_\phi(t)^2B_\phi(t)\dt\le \frac{1}{C_0},
$$
where $A_\phi$, $B_\phi$, $\mu_\phi$ are defined in~\e{n67} while $C_0$ is given by Lemma~\ref{L:2.1}.
\end{proposition}
\begin{proof}
For this proof we skip the index $\phi$ and write simply $A,B,\mu$. 

Since (see~\e{Z20'}), 
$$
\lA f(t)\rA_{\dot H^2}\le C\mu(t) B(t)^{\mez}.
$$
We then apply Proposition~\ref{P:2.6} for some fixed parameter $\beta>0$. 
Then, it follows from~\e{a2} that
\be\label{a200}
\fract A(t)+C_1\frac{B(t)}{\nu(t)^2}\leq C_2 \left( \sqrt{A(t)}+A(t) \right)\mu(t) B(t),
\ee
where
\begin{align*}
&\nu(t)=1+	\lA \partial_xf_0\rA_{L^\infty}+ C_0\int_{0}^{t}\mu(\tau)^{2}B(\tau) \dtau
+C_0 \varepsilon^{\frac{\betathree}{2}}\lA f_0\rA_{\dot H^{\frac{3}{2}}}\\&+
C_0	\varepsilon^{\frac{\betathree}{2}}\left[\sup_{\tau\in [0,t]}\left(1+\lA f(\tau)\rA_{H^{\tdm}}\right)^2\right]\log\left(2+\int_{0}^{t} \mu(\tau)^{2}B(\tau)\dtau\right)^{\mez}
\left( \int_{0}^{t} \mu(\tau)^{2}B(\tau)\dtau\right)^{\frac{1}{2}}.
\end{align*}

Given a positive number $T_0$ to be determined, 
introduce the set
$$
I(T_0)=\left\{ t\in [0,\min\{T,T_0\}]\,;\, \int_{0}^{t}\mu(\tau)^2B(\tau)\dtau\leq \frac{2}{3C_0} 
\text{ and } \sup_{0\le \tau\le t}A(\tau)\le 5M_0\right\}.
$$
We want to prove that $I(T_0)=[0,\min\{T,T_0\}]$. Since $0$ belongs to $I(T_0)$ by assumption 
on the initial data, and since $I(T_0)$ is closed, it suffices to prove that $I(T_0)$ is open. 
To do so, we consider a time $t^*\in [0,\min\{T,T_0\})$ which belongs to $I(T_0)$. 
Our goal is to prove that
$$
\int_{0}^{t^*}\mu(\tau)^2B(\tau)\dtau\leq \frac{1}{2C_0} 
\text{ and } \sup_{0\le \tau\le t^*}A(\tau)\le 4M_0.
$$
This will imply at once that $t^*$ belongs to the interior of $I(T_0)$.

As in the previous proof, we use  \eqref{f11} to write
\begin{equation}\label{f12}
 \sup_{s\in [0,t]}\lA f(s)\rA_{H^{\frac{3}{2}}}\le  \lA f_0\rA_{L^2}+(1  +C\eps^{\mez}t) \sup_{s\in [0,t]}\lA f(s)\rA_{\dot H^{\frac{3}{2}}}.
\end{equation}

It $t\le t^*$ with $t^*\in I(T_0)$, then 
\begin{align*}
\nu(t)&\le 1+	\lA \partial_xf_0\rA_{L^\infty}+ \frac{2}{3}
+C_0 \varepsilon^{\frac{\betathree}{2}}M_0\\
&\quad+
C_0	\varepsilon^{\frac{\betathree}{2}}\left(1+\lA f_0\rA_{L^2}+6M_0\right)^2 \log\left(2+
\frac{2}{3C_0}\right)^{\mez}
\left(\frac{2}{3C_0}\right)^{\frac{1}{2}}.
\end{align*}
Hence, one can define $\eps_0$ small enough, depending only on $M_0$, $\lA f_0\rA_{L^2}$ 
and the fixed parameter $\beta$, 
such that if $\eps\le \eps_0$ and 
if $t^*\in I(T_0)$, then for any $t\in [0,t^*]$, we have
$$
\nu(t)\le 2+\lA \partial_xf_0\rA_{L^\infty}.
$$
Consequently
\begin{align*}
\fract A(t)
+C_1
\frac{B(t)}{(2+\lA \partial_{x}f_0\rA_{L^\infty})^2}
\leq C_2\Big(A(t)+\sqrt{A(t)}\Big)
\mu(t)B(t).
\end{align*}
Introduce the function
$$
\mathcal{E}(r,m)\defn \sup_{\rho\ge 0}\left\{C_2
\big(\sqrt{r}+r\big)\left(\phi\left(\frac{\rho}{r}\right)\right)^{-1}
\rho-\frac{C_1}{2}\frac{\rho}{m}\right\}\cdot
$$
Then, for any $t\in [0,t^*]$, we have
\begin{align*}
\fract A(t)+\frac{C_1}{2}\frac{B(t)}{\big(2+\lA\partial_{x}f_0\rA_{L^\infty}\big)^2}
\leq\mathcal{E}\Big(A(t),\lA\partial_{x}f_0\rA_{L^\infty}\Big).
\end{align*}
Assume that the number $T_0$ satisfies
$$
T_0\le\frac{A(0)}{4\mathcal{E}\big(4A(0),\lA\partial_{x}f_0\rA_{L^\infty}\big)}\cdot
$$
Then, for any $t\leq t^*$, we get that
\begin{equation*}
\sup_{\tau\leq t} A(\tau)+\frac{C_1}{2}\frac{1}{\big(2+\lA\partial_{x}f_0\rA_{L^\infty}\big)^2}
\int_{0}^{t}B(\tau)\dtau\leq 4A(0).
\end{equation*}
In particular, for $t=t^*$, this gives
\begin{equation}\label{Z2}
\sup_{t\leq t^*} A(t)\le 4A(0),\quad 
\int_{0}^{t^*}B(t)\dt\leq \frac{8A(0)}{C_1}\big(2+\lA\partial_{x}f_0\rA_{L^\infty}\big)^2.
\end{equation}
To get the result, we must show that 
\begin{equation}\label{Z}
C_0\int_{0}^{T}\mu(t)^2 B(t)\dt\leq \frac{1}{2}.
\end{equation}
Recall that 
$$
\mu(t)=\left(\phi\left(\frac{B(t)}{A(t)}\right)\right)^{-1}.
$$
Since $\phi$ is increasing and since $A(t)\le 4A(0)$, we have
$$
\mu(t)\le \left(\phi\left(\frac{B(t)}{4A(0)}\right)\right)^{-1}.
$$
Now, we claim that the function $F\colon [0,+\infty)\to[0,+\infty)$, defined by
$$
F(r)=\left(\phi\left(\frac{r}{4A(0)}\right)\right)^{-1}r,
$$
is increasing. To see this decompose $F(r)$ under the form 
$F(r)=F_1(r)\left(F_2\left(r\right)\right)^2$ with 
$$
F_1(r)=\frac{r}{(\log(\lambda_0+r))^2}\quad F_2(r)=\frac{\log(\lambda_0+r)}{\phi(r/4A(0))}\cdot
$$
Then
\begin{align*}
&\int_{0}^{t^*}\mu(t)B(t)\dt
\leq \int_{0}^{t^*}\left(\phi\left(\frac{B(t)}{4A(0)}\right)\right)^{-2}B(t)\dt\\
&\leq \int_{0}^{t^*}\left(\phi\left(\frac{r}{4A(0)}\right)\right)^{-2}r \dt
+\int_{0}^{t^*}\left(\phi\left(\frac{r}{4A(0)}\right)\right)^{-2}B(t)\dt\\
&\overset{\eqref{Z2}}\leq t^*\left(\phi\left(\frac{r}{4A(0)}\right)\right)^{-2}r+ 
\left(\phi\left(\frac{r}{4A(0)}\right)\right)^{-2}8A(0)\left(2+\lA\partial_{x}f_0\rA_{L^\infty}\right)^2,
\end{align*}
for any $r\geq 1$. Now we successively determine two numbers $r_0>1$ and $T_0>0$ such that 
 \begin{equation}
C_0\left(\phi\left(\frac{r_0}{4A(0)}\right)\right)^{-2}8A(0)\left(2+\lA\partial_{x}f_0\rA_{L^\infty}\right)^2
=\frac{1}{4},
\end{equation}
and
\begin{equation}
T_0\left(\phi\left(\frac{r_0}{4A(0)}\right)\right)^{-2}r_0=\frac{1}{4}\cdot
\end{equation}
With this choice we get~\eqref{Z} and we obtain that $I(T_0)=[0,\min\{T,T_0\}]$, which is equivalent to the statement of the proposition.
\end{proof}

\subsection{Transfer of compactness}\label{S:Transfer}
Previously, we have proven {\em a priori\/} estimates for the spatial derivatives. 
In this paragraph, we gather results from which we will infer estimates for the time derivative as well as 
for the nonlinearity in the Muskat equation. 
These estimates serve to pass to the limit the equation (which is needed to 
regularize the solutions).

The Muskat equation~\e{n1} can be written under the form
\begin{align}\label{eqT}
\partial_tf+\D f = \mathcal{T}(f)f,
\end{align}
where $\mathcal{T}(f)$ is the operator defined by
\be\label{def:T(f)f}
\mathcal{T}(f)g = -\frac{1}{\pi}\int_\xR\left(\partial_x\Delta_\alpha g\right)
\frac{\left(\Delta_\alpha f\right)^2}{1+\left(\Delta_\alpha f\right)^2}\dalpha.
\ee

We recall the following result from Proposition~$2.3$ in \cite{Alazard-Lazar} and 
from Remark~$2.9$ and Propositions~$2.10$ and~$2.13$ in~\cite{AN1}.
\begin{proposition}\label{P:2.11}
$i)$ For all $\delta\in [0,1/2)$, there exists a constant $C>0$ such that, 
for all functions $f_1,f_2$ in $\dot{H}^{1-\delta}(\xR)\cap \dot{H}^{\tdm+\delta}(\xR)$, 
$$
\lA (\mathcal{T}(f_1)-\mathcal{T}(f_2))f_2\rA_{L^2}
\le C \lA f_1-f_2\rA_{\dot{H}^{1-\delta}}\lA f_2\rA_{\dot{H}^{\tdm+\delta}}.
$$
$ii)$ One can decompose the nonlinearity under the form
\be\label{n1T}
\mathcal{T}(f)g
=\frac{(\partial_xf)^2}{1+(\partial_xf)^2}\D g +V(f)\partial_x g+R(f,g),
\ee
where the coefficient $V(f)$ 
and the remainder term $R(f,g)$ satisfy 
the following estimates:
\begin{align}
&\lA V(f)\rA_{L^\infty}\le 
C \int_\xR \la \xi\ra \bla\hat{f}(\xi)\bra\dxi,\label{XV}\\
&\lA R(f,g)\rA_{L^2}\le C\Vert g\Vert _{\dot{H}^{\frac{3}{4}}}
\lA f\rA_{\dot{H}^{\frac{7}{4}}},\label{XR}
\end{align}
for some absolute constant $C$. 
Moreover,
\begin{equation}\label{X1}
\lA \mathcal{T}(f)f\rA_{\dot H^1}\le 
C\left(\lA f\rA_{\dot H^{\frac32}}+\lA f\rA_{\dot H^{\frac32}}^2+1
+\lA V(f)\rA_{L^\infty}\right) \lA f\rA_{\dot H^{2}},
\end{equation}
and,
\be\label{XH}
\la \big(V(f)\partial_x  g,\D g\big)\ra
\le C\Big( \lA f\rA_{\dot{H}^2}+\lA f\rA_{\dot{H}^{\frac{7}{4}}}^2\Big)
\lA g\rA_{\dot{H}^\mez}\lA g\rA_{\dot{H}^1}.
\ee
\end{proposition}

For later purpose, we need a refinement of~\e{X1}.

\begin{proposition}\label{P:2.8}
There exists a positive constant $C>0$ such that, 
for all function $f\in H^2(\xR)$,
\begin{equation}\label{n33}
\lA \mathcal{T}(f)f\rA_{\dot H^1}\leq 
C \left(1+\lA f\rA_{H^{\frac32}}\right)^2
\log\left(2+\lA f\rA_{\dot H^2}^2\right)^{\mez} \lA f\rA_{\dot H^{2}}.
\end{equation}
\end{proposition}
\begin{proof}
In view of~\e{X1} and~\e{XV}, it is sufficient to estimate the $L^1$-norm of $\la \xi\ra \hat{f}$. 
Write,
\begin{align*}
 \int_\xR |\xi| |\hat f| \dxi 
 &=\int_{|\xi|>\lambda}  |\xi|^{-1}  |\xi|^{2}|\hat f| \dxi
 +\int_{|\xi|\leq \lambda}  (|\xi|+1)^{-\mez} |\xi|(1+|\xi|)^{\mez}|\hat f| \dxi \\
 &\lesssim\left(\int_{|\xi|>\lam} \frac{1}{|\xi|^{2}} \dxi \right)^{\mez} 
 \Vert f\Vert_{\dot H^2}+ \left(\int_{|\xi|\leq \lam} \frac{1}{(|\xi|+1)} \dxi \right)^{\mez}
\left( \Vert f \Vert_{\dot H^{\frac{3}{2}}}+ \Vert f\Vert_{L^2}\right)
\\&\lesssim \lam^{-\mez} \Vert f\Vert_{\dot H^2}+ \log(1+\lam)^{\mez}
\left( \Vert f \Vert_{\dot H^{\frac{3}{2}}}+ \Vert f\Vert_{L^2}\right).
\end{align*}
Choosing $\lambda=\lA f\rA_{\dot H^2}^2$, we obtain
$$
\int_\xR |\xi| |\hat f| \dxi \lesssim 1+\log(1+\lA f\rA_{\dot H^2}^2)^{\mez}
 \left( \Vert f \Vert_{\dot H^{\frac{3}{2}}}+ \Vert f\Vert_{L^2}\right).
$$
By reporting this in \e{XV} and then using \eqref{X1}, we get 
the desired result~\e{n33}. 
\end{proof}

By using the equation~\e{eqT}, we deduce at once the following bound.
\begin{corollary}
There exists a non-decreasing function $\mathcal{F}\colon \xR^+\to\xR^+$ such that, for any $T>0$, 
any $\eps$ and 
any smooth solution $f$ in $C^1([0,T];H^\infty(\xR))$ of the Muskat equation~\e{n2}, 
if one sets
\begin{equation*}
M_\eps(T)=\sup_{t\in [0,T]}\left(\lA f(t)\rA_{\dot H^{\frac{3}{2}}}^2
+\lA f(t)\rA_{L^2}^2\right)+\int_{0}^{T}\lA f(t)\rA_{\dot H^{2}}^2\dt+\la \log (\eps)\ra^{-1} 
\int_0^T\lA f(t)\rA_{\dot{H}^{\frac{5}{2}}}^2\dt
\end{equation*}
then, 
\begin{equation}\label{n35}
\int_{0}^{T}\frac{\lA \mathcal{T}(f)f\rA_{\dot H^1}^2}{\log\big(2+\lA \mathcal{T}(f)f\rA_{\dot H^1}\big)}\dt
\leq \mathcal{F}(M_\eps(T)), 
\end{equation}
and
\begin{equation}\label{n36}
\int_{0}^{T} \frac{\lA \partial_tf\rA_{\dot H^1}^2}{\log \big(2+\Vert\partial_tf\Vert_{\dot H^1}^2)}\dt\leq \mathcal{F}(M_\eps(T)).
\end{equation}
\end{corollary}
\begin{proof}
Let $C$ be the constant given by Proposition~\ref{P:2.8} and set 
$\tilde{C}=\max\{C,1\}$. We claim that 
\begin{equation*}
\frac{\lA \mathcal{T}(f)f\rA_{\dot H^1}^2}{\log\big(2+\lA \mathcal{T}(f)f\rA_{\dot H^1}\big)}\leq 
\tilde{C}^2 \left(\lA f\rA_{\dot H^{\frac32}}
+\lA f\rA_{\dot H^{\frac32}}^2+\lA f\rA_{\dot{H}^1}+1\right)^2\lA f\rA_{\dot H^{2}}^2.
\end{equation*}
If $\lA \mathcal{T}(f)f\rA_{\dot H^1}\le \lA f\rA_{\dot H^{2}}$, then this is obvious. Otherwise, this follows at once from~\e{n33}. 
This implies~\e{n35}. 

The proof of~\e{n36} follows from similar argument, using the equation to estimate $\partial_tf$ in terms of $\mathcal{T}(f)f$. 
\end{proof}

It follows from the previous results that one can extract from the solutions of the approximate Cauchy problems~\e{n2} 
a sub-sequence converging to a solution of the Muskat equation~\e{n1}. 
Since it is rather classical, we do not include the details and refer for instance to~\cite{CG-CMP,Cordoba-Lazar-H3/2}. 

\subsection{Uniqueness}\label{S:3.5}
To prove the uniqueness of the solution to the Cauchy problem for rough initial data, 
we shall prove an estimate for the difference of two solutions. 

\begin{proposition}\label{P:2.10}
Let $T>0$ and consider two solutions $f_1,f_2$ of the Muskat equation, with initial data $f_{1,0},f_{2,0}$ respectively, 
satisfying 
$$
f_k\in C^0([0,T];\dot{W}^{1,\infty}(\xR)\cap\dot H^{\tdm}(\xR))\cap C^1([0,T];\dot H^{\mez}(\xR))\cap L^2(0,T;\dot{H}^2(\xR)),\quad k=1,2.
$$
Assume that 
\begin{equation}\label{Z105}
\sup_{t\in [0,T]}\left(\lA f_k(t)\rA_{\dot H^{\tdm}}^2+\lA f_k(t)\rA_{\dot W^{1,\infty}}^2\right)
+\int_0^{T} \lA f_k\rA_{\dot H^2}^2\dt \leq M<\infty,~~k=1,2.
\end{equation}
Then the difference $g=f_1-f_2$ is estimated by
\begin{equation}\label{Z107}
\sup_{t\in [0,T]} \Vert g(t) \Vert_{\dot{H}^\mez}\leq
\Vert g(0) \Vert_{\dot{H}^\mez}\exp\left(C(M+1)^{5} \int_0^T
\left(\lA f_1\rA_{\dot{H}^2}^2+\lA f_2\rA_{\dot{H}^2}^2\right)\dt \right).
\end{equation}
\end{proposition}
\begin{proof}
Since $\partial_tf_k+\D f_k = \mathcal{T}(f_k)f_k$, it follows from the decomposition~\eqref{n1T} 
of $\mathcal{T}(f_k)f_k$ that the difference $g=f_1-f_2$ satisfies
\begin{align*}
\partial_tg+\frac{\D g}{1+(\partial_xf_1)^2}
&=  V(f_1)\partial_x g+R(f_1,g)+\left(\mathcal{T}(f_2+g)-\mathcal{T}(f_2)\right)f_2.
\end{align*}
Since $g$ belongs to $C^1([0,T];\dot H^{\mez}(\xR))$, we may take 
the $L^2$-scalar product of this equation with $\D g$ to get
\begin{align*}
\frac{1}{2}\fract\Vert g \Vert^{2}_{\dot{H}^\mez}+\int\frac{( \D g)^2}{1+(\partial_x f_{1})^2} \dx
&\leq \left|\big(V(f_1)\partial_x g,|D| g\big)\right|+\lA R(f_1,g)\rA_{L^2}\lA g\rA_{\dot H^1}\\
&\quad+\lA \left(\mathcal{T}(f_2+g)-\mathcal{T}(f_2)\right)f_2\rA_{L^2}\lA g\rA_{\dot H^1}.
\end{align*}
It follows from Proposition~\ref{P:2.11} that
\begin{align*}
\fract\Vert g \Vert^{2}_{\dot{H}^\mez}+M^{-1}||g||_{\dot H^1}^2
&\lesssim  \left(\lA f_1\rA_{\dot{H}^2}+\lA f_1\rA_{\dot{H}^{\frac{7}{4}}}^2\right)
\lA g\rA_{\dot H^{\mez}}\lA g\rA_{\dot H^{1}}\\
&\quad+\lA f_2\rA_{\dot{H}^{\frac{7}{4}}}
\Vert g\Vert _{\dot{H}^{\frac{3}{4}}}|| g||_{\dot H^1}.
\end{align*}
By Gagliardo-Nirenberg interpolation inequality
\begin{align*}
\fract\Vert g \Vert^{2}_{\dot{H}^\mez}+M^{-1} ||g||_{\dot H^1}^2
&\lesssim  \lA f_1\rA_{\dot{H}^2}\left(1+\lA f_1\rA_{\dot{H}^{\frac{3}{2}}}\right)\lA g\rA_{\dot H^{\mez}}\lA g\rA_{\dot H^{1}}\\
&\quad+\lA f_2\rA_{\dot{H}^{2}}^{\frac{1}{2}}
\lA f_2\rA_{\dot{H}^{\frac{3}{2}}}^{\frac{1}{2}}\Vert g\Vert _{\dot{H}^{\frac{1}{2}}}^{\frac{1}{2}}
\lA g\rA_{\dot H^1}^{\frac{3}{2}}\\
&\lesssim   \lA f_1\rA_{\dot{H}^2}\left(1+M\right)\lA g\rA_{\dot H^{\mez}}\lA g\rA_{\dot H^{1}}+M^{\frac{1}{2}}\lA f_2\rA_{\dot{H}^{2}}^{\frac{1}{2}} \Vert g\Vert _{\dot{H}^{\frac{1}{2}}}^{\frac{1}{2}}\lA g\rA_{\dot H^1}^{\frac{3}{2}}.
\end{align*}
Thus, 	thanks to Holder's inequality, one gets
\begin{align*}
\fract\Vert g \Vert^{2}_{\dot{H}^\mez}+\frac{1}{2M} ||g||_{\dot H^1}^2
&\leq    C (M+1)^{5}  \left(\lA f_1\rA_{\dot{H}^2}^2+\lA f_2\rA_{\dot{H}^2}^2\right)\lA g\rA_{\dot H^{\mez}}^2
\end{align*}
which in turn implies \eqref{Z107}.
\end{proof}

\subsection{The Cauchy problem for the approximate equations}\label{S:end}
It remains to prove Proposition~\ref{P:initiale}. 

Rewrite the equation~\e{n2} under the form
\be\label{n2'}
\partial_tf-|\log(\varepsilon)|^{-1}\partial_x^2f =N_\eps(f),
\ee
with
$$
N_\eps(f)=\frac{1}{\pi}\int_\xR\frac{\partial_x\Delta_\alpha f}{1+\left(\Delta_\alpha f\right)^2}\left(1-\chi\left(\frac{|\alpha|}{\varepsilon}\right)\right)\dalpha.
$$
The next proposition shows that Equation~\e{n2'} can be seen as a sub-critical parabolic equation.

\begin{lemma}
There holds
\be\label{f3}
\begin{aligned}
\lA N_\eps(f)\rA_{\dot H^1}
\lesssim \varepsilon^{\mez}\lA f\rA_{\dot{H}^{\frac{5}{2}}}+\left(
1+\lA f\rA_{H^{\frac32}}\right)^2
\log\left(2+\lA f\rA_{\dot H^2}^2\right)^{\mez} \lA f\rA_{\dot H^{2}},
\end{aligned}
\ee
and
\be\label{f4}
\lA N_\eps(f)\rA_{L^2}\le C \left(1+\lA f\rA_{H^{\tdm}}\right)^2.
\ee
\end{lemma}
\begin{proof}
The estimate~\e{f3} follows at once from~\e{f1} and~\e{f2}. 
To prove~\e{f4}, we decompose $N_\eps(f)=-\D f+\mathcal{T}(f)f+R_\eps(f)$ 
where $\mathcal{T}(f)$ is the operator already introduced in~\S\ref{S:Transfer} and 
the remainder $R_\eps(f)$ is as defined by~\e{f10}. 
Recall from Proposition~$2.3$ in~\cite{Alazard-Lazar} that
$$
\lA \mathcal{T}(f)f\rA_{L^2}\les \lA f\rA_{\dot{H}^1}\lA f\rA_{\dot{H}^{\tdm}}.
$$
So the wanted conclusion follows from the estimate~\e{f9} for~$R_\eps(f)$.
\end{proof}
Multiply the latter equation by $(I-\Delta)^{3/2} f$ and integrate in time, to obtain
\be\label{f5}
\mez \fract \lA f\rA_{H^{\tdm}}^2
+|\log(\varepsilon)|^{-1}\lA \D f\rA_{H^{\tdm}}^2
\le \lA N_\eps(f)\rA_{H^1}\lA f\rA_{H^2}.
\ee
Recall that
\be\label{f6}
\lA N_\eps(f)\rA_{\dot H^1}
\lesssim \varepsilon^{\mez}\lA f\rA_{\dot{H}^{\frac{5}{2}}}+\left(
1+\lA f\rA_{H^{\frac32}}\right)^2
\log\left(2+\lA f\rA_{\dot H^2}^2\right)^{\mez} \lA f\rA_{\dot H^{2}},
\ee
Since $\eps^\mez\ll \la \log(\eps)\ra^{-1}$ for $\eps\ll 1$, we can absorb 
the contribution of~$\varepsilon^{\mez}\lA f\rA_{\dot{H}^{\frac{5}{2}}}$ 
in the right-hand side of \e{f6} by the left-hand side 
of 
\e{f5}. On the other hand, since $5/2>2$, one can absorb 
the contribution of the other terms by using the H\"older's 
inequality. This proves an {\em a priori} estimate for~\e{n2'}. 
We also get easily a contraction estimate 
similar to (but much simpler) the one given by Proposition~\ref{P:2.10}. 
Then by using classical tools for semi-linear equations, 
we conclude that the Cauchy problem for~\e{n2'} can be solved by standard iterative scheme.

\vfill
\begin{flushleft}
\textbf{Thomas Alazard}\\
Universit{\'e} Paris-Saclay, ENS Paris-Saclay, CNRS,\\
Centre Borelli UMR9010, avenue des Sciences, 
F-91190 Gif-sur-Yvette\\
France.

\vspace{1cm}

\textbf{Quoc-Hung Nguyen}\\
ShanghaiTech University, \\
393 Middle Huaxia Road, Pudong,\\
Shanghai, 201210,\\
China

\end{flushleft}

\end{document}